\documentclass[12pt]{amsart}
\usepackage{a4wide,enumerate,xcolor}
\usepackage{amsmath,graphicx}
\allowdisplaybreaks

\let\pa\partial
\let\na\nabla
\let\eps\varepsilon
\newcommand{\N}{{\mathbb N}}
\newcommand{\R}{{\mathbb R}}
\newcommand{\diver}{\operatorname{div}}

\newcommand{\dom}{\mathcal{D}}

\newtheorem{theorem}{Theorem}
\newtheorem{lemma}[theorem]{Lemma}

\newtheorem{remark}[theorem]{Remark}

\begin{document}

\title[Large-time asymptotics]{Large-time asymptotics for degenerate \\
cross-diffusion population models \\
with volume filling}

\author[X. Chen]{Xiuqing Chen}
\address{School of Mathematics (Zhuhai), Sun Yat-Sen University, Zhuhai 519082,
Guang\-dong Province, China}
\email{chenxiuqing@mail.sysu.edu.cn}

\author[A. J\"ungel]{Ansgar J\"ungel}
\address{Institute of Analysis and Scientific Computing, Technische Universit\"at Wien,
Wiedner Hauptstra\ss e 8--10, 1040 Wien, Austria}
\email{juengel@tuwien.ac.at}

\author[X. Lin]{Xi Lin}
\address{Department of Mathematics and Physics, Guangzhou Maritime University, Guangzhou 510765,
Guang\-dong Province, China}
\email{linxi@gzmtu.edu.cn}

\author[L. Liu]{Ling Liu}
\address{School of Mathematics (Zhuhai), Sun Yat-Sen University, Zhuhai 519082,
Guang\-dong Province, China}
\email{liuling55@mail2.sysu.edu.cn}

\date{\today}

\thanks{The first, third, and fourth authors acknowledge support from the National Natural
Science Foundation of China (NSFC), grant 11971072.
The second author acknowledges partial support from the Austrian Science Fund (FWF), grants P33010 and F65. This work has received funding from the European Research Council (ERC) under the European Union's Horizon 2020 research and innovation programme, ERC Advanced Grant no.~101018153.}

\begin{abstract}
The large-time asymptotics of the solutions to a class of degenerate parabolic cross-diffusion systems is analyzed. The equations model the interaction of an arbitrary number of population species in a bounded domain with no-flux boundary conditions. Compared to previous works, we allow for different diffusivities and degenerate nonlinearities. The proof is based on the relative entropy method, but in contrast to usual arguments, the relative entropy and entropy production are not directly related by a logarithmic Sobolev inequality. The key idea is to apply convex Sobolev inequalities to modified entropy densities including ``iterated'' degenerated functions.
\end{abstract}



\maketitle


\section{Introduction}

The aim of this note is to extend the large-time asymptotics result of \cite{ZaJu17} on multi-species cross-diffusion systems with volume-filling effects to the degenerate case. Such systems describe, for instance, the spatial segregation of population species \cite{SKT79}, chemotactic cell migration in tissues \cite{Pai09}, motility of biological cells \cite{SLH09}, or ion transport in fluid mixtures \cite{BSW12}. The main difficulties of the cross-diffusion systems are the lack of positive semidefiniteness of the diffusion matrix and the nonstandard degeneracies. The first issue was overcome by applying the boundedness-by-entropy method \cite{Jue15}, which exploits the underlying entropy (or formal gradient-flow) structure. This allows for both a global existence analysis and the proof of lower and upper bounds, without the use of a maximum principle. The second issue was handled by extending the Aubin--Lions compactness lemma \cite{ZaJu17}. However, the large-time asymptotics in \cite{ZaJu17} only holds if the problem is not degenerate. In the present note, we remove this restriction.

The evolution of the volume fraction $u_i(x,t)$ of the $i$th species is given by
\begin{align}\label{1.eq}
  & \pa_t u_i = \diver\sum_{j=1}^n A_{ij}(u)\na u_j\quad\mbox{in }\Omega,\ t>0,\ i=1,\ldots,n \\
	& \sum_{j=1}^n A_{ij}(u)\na u_j\cdot\nu = 0\quad\mbox{on }\pa\Omega, \quad
	u_i(\cdot,0) = u_i^0\quad\mbox{in }\Omega, \label{1.bic}
\end{align}
where $u_0=1-\sum_{i=1}^n u_i$ is the solvent volume fraction or the proportion of unoccupied space (depending on the application), $\Omega\subset\R^d$ ($d\ge 1$) is a bounded domain with Lipschitz boundary, $\nu$ is the exterior unit normal vector to $\pa\Omega$, and the diffusion coefficients are given by
\begin{equation}\label{1.A}
  A_{ij}(u) = D_ip_i(u)q(u_0)\delta_{ij} + D_iu_ip_i(u)q'(u_0)
	+ D_iu_i q(u_0)\frac{\pa p_i}{\pa u_j}(u),
\end{equation}
where $i,j=1,\ldots,n$, $u=(u_1,\ldots,u_n)$ is the solution vector, $D_i>0$ are the diffusivities, $\delta_{ij}$ denotes the Kronecker symbol, and $p_i$ and $q$ are smooth functions. In particular, the bounds $0\le u_i\le 1$ should hold for all $i=0,\ldots,n$. The boundary condition in \eqref{1.bic} means that the physical or biological system is isolated.
We note that equations \eqref{1.eq} and \eqref{1.A} can be written as
\begin{align}\label{1.eq2}
  \pa_t u_i = D_i\diver\bigg(u_ip_i(u)q(u_0)\na\log\frac{u_ip_i(u)}{q(u_0)}\bigg)
  = D_i\diver\bigg(q(u_0)^2\na{\frac{u_ip_i(u)}{q(u_0)}}
  \bigg).
\end{align}%
In some applications, drift or reaction terms need to be added; see, e.g., \cite{BDPS10,GeJu18} for systems with drift terms and \cite{DJT20} for reaction rates.

Equations \eqref{1.eq} and \eqref{1.A} can be formally derived from a random-walk lattice model in the diffusion limit \cite[Appendix A]{ZaJu17}. The functions $p_i$ and $q$ are related to the transition rates of the lattice model with $p_i$ measuring the occupancy and $q$ measuring the non-occupancy. This class of systems contains the population model of Shigesada, Kawasaki, and Teramoto \cite{SKT79} (if $p_i$ is a linear function and $q=1$) and Nernst--Planck-type equations accounting for finite ion sizes (if $p_i=1$ and $q(u_0)=u_0$; see \cite{GeJu18}).
In this note, we consider the degenerate case $q'(0)=0$ and assume that there exists a smooth function $\chi$ such that $p_i=\exp(\pa\chi/\pa u_i)$ to guarantee an entropy structure via the entropy density
\begin{equation}\label{1.h}
  h(u) = \sum_{i=1}^n (u_i(\log u_i-1)+1) + \int_1^{u_0}\log q(s)ds + \chi(u),
\end{equation}
where $u\in\dom:=\{u\in(0,1)^n:\sum_{i=1}^n u_i<1\}$.

There exist other approaches to model volume filling. The finite particle size may be taken into account by adding cross-diffusion terms of the type $u_i\na\sum_{j=1}^n b_{ij}u_j$ to the standard Nernst--Planck flux \cite{Hsi19} or by using the Bikerman-type flux $J_i=-D_i(\na u_i-u_i\na\log u_0)$ in the mass conservation equation $\pa_t u_i+\diver J_i=0$ \cite{Bik42}.

The global existence of bounded weak solutions to \eqref{1.eq}--\eqref{1.A} was shown in \cite[Theorem 1]{ZaJu17} assuming $D_i=1$ for $i=1,\ldots,n$ and the following conditions:
\begin{itemize}
\item[\bf (H1)] Domain: $\Omega\subset\R^d$ ($d\ge 1$) is a bounded convex domain with Lipschitz boundary, $T>0$. Set $\dom=\{u\in(0,1)^n:\sum_{i=1}^n u_i<1\}$ and $\Omega_T=\Omega\times(0,T)$.
\item[\bf (H2)] Initial datum: $u^0(x)\in\dom$ for a.e.\ $x\in\Omega$ and $h(u^0)\in L^1(\Omega)$.
\item[\bf (H3)] Functions $p_i$: $p_i=\exp(\pa\chi/\pa u_i)$, where $\chi\in C^3(\overline{\dom})$ is convex.
\item[\bf (H4)] Function $q$: $q\in C^3([0,1])$ satisfies $q(0)=0$, $q(1)=1$, $q'(0)\ge 0$ and $q(s)>0$, $q'(s)>0$ for all $0<s\le 1$.
\end{itemize}

The convexity of $\Omega$ in Hypothesis (H1) is used for the convex Sobolev inequality; see Lemma \ref{lem.csi} below. For generalized Nernst--Planck systems with $p_i=\mbox{const.}$, we may choose $\chi(u)=\sum_{i=1}^n u_i$, which satisfies Hypothesis (H3). Moreover, if $p_i(u)=P_i(u_i)$ for some functions $P_i:[0,1]\to[0,\infty)$, condition $p_i=\exp(\pa\chi/\pa u_i)$ is satisfied with $\chi(u)=\sum_{i=1}^n\chi_i(u_i)$ and $\chi_i(s)=\int_0^s\log P_i(\tau)d\tau$. The functions $q(s)=s^\alpha$ with $\alpha\ge 1$ satisfy Hypothesis (H4).

We claim that the existence result also holds for arbitrary $D_i>0$. Indeed, it is sufficient to define $\widetilde\chi(u)=\chi(u)+\sum_{j=1}^n u_j\log D_j$, since $\exp(\pa\widetilde\chi/\pa u_i)=D_i\exp(\pa\chi/\pa u_i)=D_ip_i$, and we can apply Theorem 1 in \cite{ZaJu17} with $\widetilde\chi$. We observe that the condition $q'(s)/q(s)\ge c_1>0$ in \cite{ZaJu17} is not needed for the existence analysis.

The weak solution $u=(u_1,\ldots,u_n)$ to \eqref{1.eq}--\eqref{1.A} satisfies $u(x,t)\in\dom$ for a.e.\ $(x,t)\in\Omega_T$, mass conservation, the regularity
\begin{align*}
  & \sqrt{q(u_0)},\ \sqrt{q(u_0)u_i}\in L^2(0,T;H^1(\Omega)), \quad
	\sqrt{q(u_0)}\na u_i\in L^2(\Omega_T), \\
	& \pa_t u_i\in L^2(0,T;H^1(\Omega)') \quad\mbox{for }i=,1\ldots,n,
\end{align*}
and the weak formulation
\begin{align*}
  \int_0^T\langle\pa_t u_i,\phi_i\rangle dt
	= -\int_0^T\int_\Omega D_i \sqrt{q(u_0)}\big[\na\big(u_ip_i(u)\sqrt{q(u_0)}\big)
	- 3u_ip_i(u)\na\sqrt{q(u_0)}\big]\cdot\na\phi_i dxdt
\end{align*}
for all $\phi_i\in L^2(0,T;H^1(\Omega))$, $i=1,\ldots,n$, where $\langle\cdot,\cdot\rangle$ denotes the duality product of $H^1(\Omega)'$ and $H^1(\Omega)$. Moreover, the initial datum in \eqref{1.bic} is satisfied in the sense of $H^1(\Omega)'$ and the entropy inequality
\begin{align}\label{1.ei}
  \int_\Omega h(u(t))dx
	+ c_0\int_s^t\int_\Omega\bigg(q(u_0)\sum_{i=1}^n|\na\sqrt{u_i}|^2
	+ |\na\sqrt{q(u_0)}|^2\bigg)dxdr \le \int_\Omega h(u(s))dx,
\end{align}
holds for $0\le s<t$, $t>0$ for some $c_0>0$ depending on $D_i$, $p_i$, and $q$, recalling definition \eqref{1.h} of $h(u)$. The $L^\infty(\Omega_T)$ bound for $u_i$ and the $L^2(\Omega_T)$ for $\sqrt{q(u_0)}\na u_i$ imply that $\na(u_ip_i(u)\sqrt{q(u_0)})\in L^2(\Omega_T)$, so that the weak formulation is well defined.

Our main result is the convergence of the solutions to \eqref{1.eq}--\eqref{1.A} towards the constant steady state
$$
  u_i^\infty = \frac{1}{|\Omega|}\int_\Omega u_i^0dx
  \quad\mbox{for }i=1,\ldots,n,
  \quad u_0^\infty= 1 - \sum_{i=1}^n u_i^\infty
$$
for large times under the following additional hypothesis:
\begin{itemize}
\item[\bf (H5)] $q$ is convex, $q/q'$ is concave, and there exist $\beta\in[0,1]$, $c_1>0$ such that
$$
  \lim_{s\to 0} \frac{s^\beta q'(s)}{q(s)}=c_1>0.
$$
\end{itemize}
Examples of functions satisfying Hypothesis (H5) are $q(s)=s^\alpha$ with $\alpha\ge 1$.
The convergence (with exponential decay rate) was proved in \cite{ZaJu17} for the nondegenerate case $q'(0)>0$ only. In the degenerate situation $q'(0)=0$, the numerical results of \cite{GeJu18} indicate that exponential rates cannot be expected. Therefore, we show the convergence without rate.

\begin{theorem}[Large-time asymptotics]\label{thm.time}
Let Hypotheses (H1)--(H5) hold and let $u=(u_1,$ $\ldots,u_n)$ be a weak solution to \eqref{1.eq}--\eqref{1.A} satisfying the entropy inequality \eqref{1.ei}. Then $u_i(t)\to u_i^\infty$ strongly in $L^p(\Omega)$ as $t\to\infty$ for all $i=1,\ldots,n$ and $1\le p<\infty$.
\end{theorem}

The idea of the proof is to exploit, as in \cite{ZaJu17}, the relative entropy density (or Bregman distance)
\begin{equation}\label{1.hstar}
  h^*(u|u^\infty) = h(u) - h(u^\infty) - h'(u^\infty)\cdot(u-u^\infty),
\end{equation}
where $u=(u_1,\ldots,u_n)$ is the weak solution to \eqref{1.eq}--\eqref{1.A}. The entropy inequality implies that
$$
  \frac{dh^*}{dt}(u|u^\infty)
	+ \frac{c_0}{2}\int_\Omega\sum_{i=1}^n|\na\sqrt{q(u_0)u_i}|^2 dx \le 0.
$$
Unfortunately, the entropy production integral cannot be estimated in terms of the relative entropy directly by applying a logarithmic Sobolev inequality to $u_i$. We overcome this issue by using two ideas.

First, we apply the logarithmic Sobolev inequality to $\sqrt{q(u_0)u_i}$,
$$
  \int_\Omega q(u_0)u_i\log\frac{q(u_0)u_i}{|\Omega|^{-1}\int_\Omega q(u_0)u_idx}dx
	\le C\int_\Omega|\na\sqrt{q(u_0)u_i}|^2 dx.
$$
The idea is to relate the integrand of the left-hand side to the relative entropy part
$h_1^*(u|u^\infty)=\sum_{i=1}^n(u_i\log(u_i/u_i^\infty)-u_i+u_i^\infty)dx$. For this, we define
$$
  f_1(u) = \sum_{i=1}^n\bigg(q(u_0)u_i\log\frac{q(u_0)u_i}{|\Omega|^{-1}\int_\Omega q(u_0)u_idx}
	- q(u_0)u_i + \frac{1}{|\Omega|}\int_\Omega q(u_0)u_idx\bigg).
$$
Since $\int_0^\infty\int_\Omega|\na\sqrt{q(u_0)u_i}|^2 dxdt<\infty$, we also have
$\int_0^\infty\int_\Omega f_1(u)dxdt<\infty$, and there exists a subsequence $t_k\to\infty$
such that $f_1(u(t_k))\to 0$. The key result is the limit (see Lemma \ref{lem.key})
$$
  \lim_{t_k\to\infty}\bigg(\frac{f_1(u(t_k))}{|\Omega|^{-1}\int_\Omega q(u_0(t_k))dx}
	- h_1^*(u(t_k)|u^\infty)\bigg) = 0.
$$
This result shows that $h_1^*(u(t_k)|u^\infty)\to 0$ as $t_k\to\infty$.

Second, instead of the part $h_2^*(u|u^\infty)=\int_{u_0^\infty}^{u_0}\log(q(s)/q(u_0^\infty))ds$ of the relative entropy density, we analyze the function
$$
  f_2(u_0) = \int_{\bar{q}}^{q(u_0)}\log\frac{q(s)}{q(\bar{q})}ds,
$$
where $\bar{q}:=|\Omega|^{-1}\int_\Omega q(u_0)dx$, which can be seen as an ``iterated'' version of $h_2^*(u|u^\infty)$, since it involves $q\circ q$ instead of $q$. Then an application of the convex Sobolev inequality yields a bound for the integral over $|\na\sqrt{q(u_0)}|^2$ without the need of condition $q'(0)>0$; see Remark \ref{rem.f2} for details. It follows from $\int_0^\infty\int_\Omega|\na\sqrt{q(u_0)}|^2 dxdt<\infty$ that $\int_0^\infty\int_\Omega f_2(u)dxdt<\infty$, and there exists a subsequence $t_k\to\infty$ such that $f_2(u(t_k))\to 0$.

The convergences $f_1(u(t_k))\to 0$ and $f_2(u(t_k))\to 0$ as well as the monotonicity of the entropy imply that $h^*(u(t_k)|u^\infty) \to 0$ pointwise. The monotonicity of $t\mapsto \int_\Omega h^*(u(t)|u^\infty)dx$ then implies the convergence for all sequences $t\to\infty$ and finally $u_i(t)\to u_i^\infty$ strongly in $L^2(\Omega)$.

To conclude the introduction, we mention some results on the large-time asymptotics for diffusion systems. Exponential equilibration rates in $L^p(\Omega)$ norms were shown for reaction-diffusion systems in \cite{DeFe06,DFM08}, for electro-reaction-diffusion systems in \cite{GlHu97}, and for Maxwell--Stefan systems for chemically reacting fluids in \cite{DJT20,JuSt13}. The convergence to equilibrium was proved for Shigesada--Kawasaki--Teramoto cross-diffusion systems without rate in \cite{Shi06}, for instance. All these results concern nondegenerate diffusion equations. The work \cite{BDPS10} is concerned with the large-time asymptotics for systems like \eqref{1.eq} with $D_i=p_i=1$ and $q(u_0)=u_0$ without rate.
The asymptotics for solutions to Poisson--Nernst--Planck-type equations with quadratic nonlinearity was investigated in \cite{Zin16} using Wasserstein techniques. Decay rates for degenerate diffusion systems without cross-diffusion terms were derived in \cite{CJMTU01}.
An extension of our results to cross-diffusion systems with drift or reactions seems delicate; see Remark \ref{rem.drift} for drift terms and \cite{DJT20} for cross-diffusion systems with reversible reactions.


\section{Proof of Theorem \ref{thm.time}}

We first recall the convex Sobolev inequality; see \cite[Lemma 11]{ZaJu17}.

\begin{lemma}[Convex Sobolev inequality]\label{lem.csi}
Let $\Omega\subset\R^d$ ($d\ge 1$) be a convex domain and let $g\in C^4(\R)$ be convex such that $1/g''$ is concave. Then there exists $C_S>0$ such that for all $v\in L^1(\Omega)$ such that $g(v)$, $g''(v)|\na u|^2\in L^1(\Omega)$,
$$
  \frac{1}{|\Omega|}\int_\Omega g(v)dx - g\bigg(\frac{1}{|\Omega|}\int_\Omega vdx\bigg)
	\le \frac{C_S}{|\Omega|}\int_\Omega g''(v)|\na v|^2 dx.
$$
\end{lemma}

The logarithmic Sobolev inequality is obtained for the choice $g(v)=v(\log v-1)+1$:
\begin{equation}\label{3.lsi}
  \int_\Omega v\log\frac{v}{|\Omega|^{-1}\int_\Omega vdx}dx
	\le 4C_S\int_\Omega|\na\sqrt{v}|^2 dx
\end{equation}
and for functions $\sqrt{v}\in H^1(\Omega)$.

Since $h(u^\infty)$ is independent of time (because of mass conservation), the entropy inequality \eqref{1.ei} implies the relative entropy inequality
\begin{align}\label{3.ei}
  \int_\Omega h^*(u(t)|u^\infty)dx &+ c_0\int_s^t\int_\Omega
	\bigg(q(u_0)\sum_{i=1}^n|\na\sqrt{u_i}|^2 + |\na\sqrt{q(u_0)}|^2\bigg)dxdr \\
	&\le \int_\Omega h^*(u(s)|u^\infty)dx \nonumber
\end{align}
for $0\le s<t$ and $t>0$, where $h^*(u|u^\infty)$ is defined in \eqref{1.hstar}. As mentioned in the introduction, we cannot apply the logarithmic Sobolev inequality \eqref{3.lsi} with $v=u_i$ since $q(u_0)=0$ for $u_0=0$. Instead we apply this inequality to $v=q(u_0)u_i$.

We split the relative entropy density $h^*$ into three parts, $h^*=h_1^* + h_2^* + h_3^*$, where
\begin{align*}
  h_1^*(u|u^\infty) &= \sum_{i=1}^n\bigg(u_i\log\frac{u_i}{u_i^\infty}
	- u_i + u_i^\infty\bigg), \\
	h_2^*(u|u^\infty) &= \int_{u_0^\infty}^{u_0}\log\frac{q(s)}{q(u_0^\infty)}ds, \\
	h_3^*(u|u^\infty) &= \chi(u) - \chi(u^\infty) - \sum_{i=1}^n(u_i-u_i^\infty)\log p_i(u^\infty),
\end{align*}
where $\chi$ is introduced in Hypothesis (H3).

\begin{lemma}
The functions $h_i^*(\cdot|u^\infty)$, $i=1,2,3$, are nonnegative and bounded on $\overline\dom$.
\end{lemma}

\begin{proof}
The function $h_1^*$ is bounded since $u_i\mapsto u_i\log u_i$ is bounded for $0\le u_i\le 1$, and $h_3^*$ is bounded thanks to Hypothesis (H3) on $p_i$. Integrating by parts in $h_2^*(u|u^\infty)$ and observing that $u_0\log q(u_0)\le 0$, we find that
\begin{align}\label{3.aux}
  h_2^*(u|u^\infty) = u_0\log\frac{q(u_0)}{q(u_0^\infty)} - \int_{u_0^\infty}^{u_0}
	s\frac{q'(s)}{q(s)}ds \le -\log q(u_0^\infty) + \int_0^1 s\frac{q'(s)}{q(s)}ds.
\end{align}
By Hypothesis (H5), $\lim_{s\to 0} sq'(s)/q(s) = \lim_{s\to 0}s^{1-\beta}\cdot s^\beta q'(s)/q(s)$ is finite (here, we need $\beta\le 1$). Therefore, $s\mapsto sq'(s)/q(s)$ is bounded on $[0,\delta]$ for some $\delta>0$. On the other hand, $s\mapsto sq'(s)/q(s)$ is also bounded on $[\delta,1]$ since this function is continuous  and $q(s)>0$ for $s>0$ is nondecreasing. This shows that $\int_0^1 (sq'(s)/q(s))ds$ is bounded, proving the claim.
\end{proof}

\subsection{Study of some auxiliary functions}

The study of the large-time behavior is based on the analysis of the two functions
\begin{align}\label{3.f}
  f_1(u) = \sum_{i=1}^n\bigg(q(u_0)u_i\log\frac{q(u_0)u_i}{\bar{q}_i} - q(u_0)u_i
	+ \bar{q}_i\bigg), \quad
	f_2(u_0) = \int_{\bar{q}}^{q(u_0)}\log\frac{q(s)}{q(\bar{q})}ds,
\end{align}
for $u\in\overline{\dom}$, where
\begin{equation}\label{3.qbar}
  \bar{q} = \frac{1}{|\Omega|}\int_\Omega q(u_0)dx, \quad
	\bar{q}_i = \frac{1}{|\Omega|}\int_\Omega q(u_0)u_idx.
\end{equation}

\begin{lemma}\label{lem.f}
The function $f_1$ is nonnegative, and the function $f_2$ is nonnegative and bounded on $\overline\dom$.
\end{lemma}

\begin{proof}
Set $z=q(u_0)u_i/\bar{q}_i$ and let $u\in\overline{\dom}$. Then
$$
  f_1(u) = \sum_{i=1}^n\bar{q}_i(z\log z - z+1)\ge 0,
$$
proving the first claim. To show the nonnegativity of $f_2$, we distinguish two cases. If $q(u_0(x,t))\ge\bar{q}$ at some $(x,t)\in\Omega_T$, then $\log(q(s)/q(\bar{q}))\ge 0$ for any $\bar{q}\le s\le q(u_0(x,t))$ and consequently $f_2(u(x,t))\ge 0$. If $q(u_0(x,t))<\bar{q}$, we have
$\log(q(s)/q(\bar{q})) < 0$ for $q(u_0(x,t))\le s\le\bar{q}$ and
$f_2(u_0(x,t)) = \int_{q(u_0(x,t))}^{\bar{q}}\log(q(\bar{q})/q(s))ds \ge 0$.

It remains to show that $f_2$ is bounded. Since $q$ is convex, Jensen's inequality shows that $\bar{q}\ge q(|\Omega|^{-1}\int_\Omega u_0dx)=q(u_0^\infty)$.
Then, using integration by parts and arguing as in \eqref{3.aux},
\begin{align*}
  f_2(u_0) &= q(u_0)\log\frac{q(q(u_0))}{q(\bar{q})}
	- \int_{\bar{q}}^{q(u_0)} s\frac{q'(s)}{q(s)}ds
	\le -q(u_0)\log q(\bar{q}) + \int_0^1 s\frac{q'(s)}{q(s)}ds \\
	&\le -\log q(q(u_0^\infty)) + \int_0^1 s\frac{q'(s)}{q(s)}ds.
\end{align*}
We already showed above that the last integral is bounded. This finishes the proof.
\end{proof}


\subsection{Convergence of $f_1$ and $f_2$}

\begin{lemma}\label{lem.convf}
It holds for a.e.\ $x\in\Omega$, $s\in(0,1]$ that
$$
  \lim_{N\to\infty}f_1(u(x,s+N)) = 0, \quad \lim_{N\to\infty}f_2(u_0(x,s+N)) = 0.
$$
\end{lemma}

\begin{proof}
The idea is to exploit the boundedness of the entropy production integrated over $t\in(0,\infty)$. First, we consider $f_1$. We know from \eqref{3.ei} for $s=0$ and $t\to\infty$ that
\begin{equation}\label{3.infty}
  c_0\int_0^\infty\int_\Omega\bigg(q(u_0)\sum_{i=1}^n|\na\sqrt{u_i}|^2
	+ |\na\sqrt{q(u_0)}|^2\bigg)dxdt \le \int_\Omega h^*(u^0|u^\infty)dx.
\end{equation}
Thus, in view of $q(u_0)u_i\ge 0$ and
\begin{align*}
  |\na\sqrt{q(u_0)u_i}|^2 &= q(u_0)|\na\sqrt{u_i}|^2
	+ 2\sqrt{q(u_0)u_i}\na\sqrt{q(u_0)}\cdot\na\sqrt{u_i} + u_i|\na\sqrt{q(u_0)}|^2 \\
	&\le 2q(u_0)|\na\sqrt{u_i}|^2 + 2|\na\sqrt{q(u_0)}|^2,
\end{align*}
it follows for a constant $C>0$ being independent of time that
\begin{equation*}
  \int_0^\infty\int_\Omega|\na\sqrt{q(u_0)u_i}|^2 dxdt \le C.
\end{equation*}
Furthermore, by the logarithmic Sobolev inequality \eqref{3.lsi}, applied to $v=q(u_0)u_i$,
$$
  \int_0^\infty\int_\Omega q(u_0)u_i\log\frac{q(u_0)u_i}{\bar{q}_i}dx
	\le C\int_0^\infty\int_\Omega|\na\sqrt{q(u_0)u_i}|^2 dx \le C,
$$
recalling definition \eqref{3.qbar} of $\bar{q}_i$. Taking into account definition \eqref{3.f} of $f_1$, we see that
$$
  \int_0^\infty\int_\Omega f_1(u(x,t))dxds
	= \sum_{N=0}^{\infty}\int_0^1\int_\Omega f_1(u(x,s+N))dx ds < \infty.
$$
Therefore, the sequence $N\mapsto \int_0^1\int_\Omega f_1(u(\cdot,s+N))dx ds$ converges to zero,
$$
  \lim_{N\to\infty}f_1(u(x,s+N)) = 0\quad\mbox{for a.e. }x\in\Omega,\ s\in(0,1].
$$

Next, we prove the limit for $f_2$. For any fixed $t>0$, we introduce the nonnegative function
$$
  f(s;t) = \int_{\bar{q}(t)}^s\log\frac{q(\sigma)}{q(\bar{q}(t))}d\sigma, \quad 0<s\le 1.
$$
By Lemma \ref{lem.f}, $x\mapsto f(q(u_0(x,t));t) = f_2(u(x,t))$ is integrable in $\Omega$
for any fixed $t>0$. Moreover, $f(\cdot,t)$ is twice differentiable in $(0,1)$:
$$
  \frac{df}{ds}(s;t) = \log\frac{q(s)}{q(\bar{q}(t))}, \quad
	\frac{d^2 f}{ds^2}(s;t) = \frac{q'(s)}{q(s)} > 0.
$$
We infer from the positivity of $d^2f/ds^2$ that $f(\cdot,t)$ is convex. By Hypothesis (H5), $(d^2f/ds^2)^{-1} = q/q'$ is concave. Thus, the assumptions of the convex Sobolev inequality (Lemma \ref{lem.csi}) are satisfied for $f(q(u_0(x,t));t)$:
\begin{align*}
  \frac{1}{|\Omega|}\int_\Omega & f(q(u_0(x,t));t)dx
	- f\bigg(\frac{1}{|\Omega|}\int_\Omega q(u_0(x,t))dx;t\bigg) \\
	&\le C(\Omega)\int_\Omega \frac{q'(q(u_0(x,t)))}{q(q(u_0(x,t)))}|\na q(u_0)|^2 dx.
\end{align*}
Hence, since $f(\bar{q}(t);t)=0$ by definition and recalling that $f(q(u_0(x,t));t) = f_2(u_0(x,t))$, the previous inequality becomes
\begin{equation}\label{2.f2est}
  \int_\Omega f_2(u_0)dx
	\le C(\Omega)\int_\Omega\frac{q(u_0)q'(q(u_0))}{q(q(u_0))}\frac{|\na q(u_0)|^2}{q(u_0)}dx
	\le C\int_\Omega|\na\sqrt{q(u_0)}|^2 dx,
\end{equation}
where we used Hypothesis (H5) to infer that
$$
  \frac{s q'(s)}{q(s)} = s^{1-\beta}\frac{s^\beta q'(s)}{q(s)}\quad\mbox{with }
	s = q(u_0)
$$
is bounded in $[0,1]$. By \eqref{3.infty}, the integrated entropy dissipation is finite:
$$
  \int_0^\infty\int_\Omega f_2(u_0)dxdt \le C\int_0^\infty\int_\Omega|\na\sqrt{q(u_0)}|^2
	dxdt \le C.
$$
Therefore, arguing as for the function $f_1$, we obtain
$\lim_{N\to\infty}f_2(u_0(x,s+N)) = 0$ for a.e.\ $x\in\Omega$, $s\in(0,1]$,
which finishes the proof.
\end{proof}
\begin{remark}\label{rem.f2}\rm
In the nondegenerate case $q'(0)>0$, it was shown in \cite[Section 5]{ZaJu17} that $t\mapsto
h_2^*(u(t)|u^\infty)$ converges to zero exponentially fast. Indeed, applying the convex Sobolev inequality similarly as in the previous proof,
\begin{equation}\label{2.h2est}
  \int_\Omega h_2^*(u|u^\infty)dx \le C\int_\Omega\frac{q'(u_0)}{q(u_0)}|\na u_0|^2 dx
  = 4C\int_\Omega\frac{|\na\sqrt{q(u_0)}|^2}{q'(u_0)} dx,
\end{equation}
and we conclude from the entropy inequality \eqref{1.ei} and Gronwall's lemma. Since we allow for $q'(0)=0$, this argument cannot be used here. We solve this issue by considering the ``iterated'' function $f_2$ involving $q\circ q$ and assuming that $s\mapsto sq'(s)/q(s)$ is bounded; see \eqref{2.f2est}. The iterated use of $q$ gives the term $|\na\sqrt{q(u_0)}|^2$ in \eqref{2.f2est} without requiring the nondegeneracy condition $q'(0)>0$.
\qed\end{remark}

A consequence of the limit for $f_2$ is the following result.

\begin{lemma}\label{lem.conv1}
If $\lim_{N\to\infty}f_2(u_0(x,s+N))=0$ for some $x\in\Omega$, $s\in(0,1]$ then
$$
  \lim_{N\to\infty}\frac{q(u_0(x,s+N))}{\bar{q}(s+N)} = 1.
$$
\end{lemma}

\begin{proof}
We write $u_i^N:=u_i(x,s+N)$ and $\bar{q}^N=\bar{q}(s+N)$
to simplify the notation. We recall from Lemma \ref{lem.f} that $f_2$ is nonnegative and change the variable $\sigma=s/\bar{q}^N$ in the integral:
\begin{align*}
  f_2(u_0^N) &= \int_{\bar{q}^N}^{q(u_0^N)}\log\frac{q(s)}{q(\bar{q}^N)}ds
	= \bar{q}^N\int_1^{q(u_0^N)/\bar{q}^N}\log\frac{q(\bar{q}^N\sigma)}{q(\bar{q}^N)}d\sigma \\
	&\ge q(u_0^\infty)\int_1^{q(u_0^N)/\bar{q}^N}
	\log\frac{q(\bar{q}^N\sigma)}{q(\bar{q}^N)}d\sigma,
\end{align*}
where we used Jensen's inequality to find that $\bar{q}^N\ge q(|\Omega|^{-1}\int_\Omega u_0^N dx) = q(u_0^\infty)$. Moreover, since $\bar{q}^N\le 1$,
$$
  q(u_0^\infty)\int_1^{q(u_0^N)/\bar{q}^N}
	\log\frac{q(\bar{q}^N\sigma)}{q(\bar{q}^N)}d\sigma
  \le f_2(u_0^N) \le \int_1^{q(u_0^N)/\bar{q}^N}
	\log\frac{q(\bar{q}^N\sigma)}{q(\bar{q}^N)}d\sigma.
$$
This shows that $\lim_{N\to\infty}f_2(u_0^N)=0$ if and only if
\begin{equation}\label{3.con}
  \lim_{N\to\infty}\int_1^{q(u_0^N)/\bar{q}^N}
	\log\frac{q(\bar{q}^N\sigma)}{q(\bar{q}^N)}d\sigma = 0.
\end{equation}

Set $A:=\{(x,s)\in\Omega\times(0,1]:\lim_{N\to\infty} f_2(u_0(x,s+N))=0\}$. We want to show that $\lim_{N\to\infty} q(u_0^N)/\bar{q}^N=1$ for $(x,s)\in A$. If not, there exist $(x_0,s_0)\in A$ and $\eps_0>0$ such that either
$$
  \frac{q(u_0^N)}{\bar{q}^N} > 1+\eps_0 \quad\mbox{or}\quad
	\frac{q(u_0^N)}{\bar{q}^N} < 1-\eps_0\quad\mbox{for all }N\in\N.
$$

In the former case, we have $q(\bar{q}^N\sigma)\ge q(\bar{q}^N(1+\eps_0/2))$
for $\sigma\ge 1+\eps_0/2$, since $q$ is increasing, and therefore,
\begin{equation}\label{3.aux2}
  \int_1^{q(u_0^N)/\bar{q}^N}\log\frac{q(\bar{q}^N\sigma)}{q(\bar{q}^N)}d\sigma
	\ge \int_{1+\eps_0/2}^{1+\eps_0}\log\frac{q(\bar{q}^N(1+\eps_0/2))}{q(\bar{q}^N)}d\sigma.
\end{equation}
Using the convexity of $q$, a Taylor expansion shows that
$q(\bar{q}^N + \bar{q}^N\eps_0/2) \ge q(\bar{q}^N) + q'(\bar{q}^N)\bar{q}^N\eps_0/2$.
Then the integrand of the previous integral can be estimated according to
$$
  \log\bigg(\frac{q(\bar{q}^N(1+\eps_0/2))}{q(\bar{q}^N)}\bigg)
	\ge \log\bigg(1 + \frac{q'(\bar{q}^N)}{q(\bar{q}^N)}\bar{q}^N\frac{\eps_0}{2}\bigg)
	\ge \log\bigg(1 + c_0q(u_0^\infty)^{1-\beta}\frac{\eps_0}{2}\bigg),
$$
where we used Hypothesis (H5) and $\bar{q}^N\ge q(u_0^\infty)$ in the last step, and $c_0>0$ is some constant. As the right-hand side is independent of $\sigma$, we infer from \eqref{3.aux2} that
$$
  \int_1^{q(u_0^N)/\bar{q}^N}\log\frac{q(\bar{q}^N\sigma)}{q(\bar{q}^N)}d\sigma
	\ge \frac{\eps_0}{2}\log\bigg(1 + c_0q(u_0^\infty)^{1-\beta}\frac{\eps_0}{2}\bigg).
$$

In the latter case $q(u_0^N)/\bar{q}^N<1-\eps_0$, we estimate as
\begin{align*}
  \int_1^{q(u_0^N)/\bar{q}^N}\log\frac{q(\bar{q}^N\sigma)}{q(\bar{q}^N)}d\sigma
	&= \int_{q(u_0^N)/\bar{q}^N}^1\log\frac{q(\bar{q}^N)}{q(\bar{q}^N\sigma)}d\sigma \\
	&\ge \int_{1-\eps_0}^{1-\eps_0/2}\log\frac{q(\bar{q}^N)}{q(\bar{q}^N(1-\eps_0/2)}d\sigma.
\end{align*}
We apply again a Taylor expansion, similarly as in the first case,
$$
  q(\bar{q}^N) = q\bigg(\bar{q}^N\bigg(1-\frac{\eps_0}{2}\bigg)
	+ \frac{\eps_0}{2}\bar{q}^N\bigg)
	\ge q\bigg(\bar{q}^N\bigg(1-\frac{\eps_0}{2}\bigg)\bigg)
	+ q'\bigg(\bar{q}^N\bigg(1-\frac{\eps_0}{2}\bigg)\bigg)\frac{\eps_0}{2}\bar{q}^N,
$$
which leads to
$$
  \log\frac{q(\bar{q}^N)}{q(\bar{q}^N(1-\eps_0/2)}
	\ge \log\bigg(1 + \frac{q'(\bar{q}^N(1-\eps_0/2))}{q(\bar{q}^N(1-\eps_0/2))}
	\frac{\eps_0}{2}\bar{q}^N\bigg)
	\ge \log\bigg(1 + c_0q(u_0^\infty)^{1-\beta}\frac{\eps_0}{2}\bigg).
$$
Thus, in both cases,
$$
  \int_1^{q(u_0^N)/\bar{q}^N}\log\frac{q(\bar{q}^N\sigma)}{q(\bar{q}^N)}d\sigma > 0
	\quad\mbox{uniformly in }N\in\N,
$$
which contradicts \eqref{3.con} and consequently $\lim_{N\to\infty}f_2(u_0^N)=0$.
\end{proof}


\subsection{Key lemma}

We show that $f_1(u(\cdot,s+N))/\bar{q}(s+N)$ and $h_1^*(u(\cdot,s+N)|u^\infty)$ are close for sufficiently large $N\in\N$. The following lemma is the key of the proof.

\begin{lemma}\label{lem.key}
For a.e.\ $x\in\Omega$, $s\in(0,1]$, it holds that
$$
  \lim_{N\to\infty}\bigg(\frac{f_1(u(x,s+N))}{\bar{q}(s+N)}
	- h_1^*(u(x,s+N)|u^\infty)\bigg) = 0.
$$
\end{lemma}

\begin{proof}
We set $u^N:=u(\cdot,s+N)$, $\bar{q}^N=\bar{q}(s+N)$, and $\bar{q}_i^N=|\Omega|^{-1}\int_\Omega q(u_0^N)u_i^Ndx$. Inserting definition \eqref{3.f} of $f_1$, the lemma is proved if we can show that for any $i=1,\ldots,n$,
\begin{align}\label{3.aux3}
  0 &= \lim_{N\to\infty}\bigg(\frac{q(u_0^N)}{\bar{q}^N}u_i^N
	\log\frac{q(u_0^N)u_i^N}{\bar{q}_i^N}
	- \frac{q(u_0^N)}{\bar{q}^N}u_i^N + \frac{\bar{q}_i^N}{\bar{q}^N}
	- u_i^N\log\frac{u_i^N}{u_i^\infty}
	+ u_i^N - u_i^\infty\bigg) \\
	&= \lim_{N\to\infty}\bigg\{\bigg(\frac{q(u_0^N)}{\bar{q}^N}u_i^N
	\log\frac{q(u_0^N)u_i^N}{\bar{q}_i^N} - u_i^N\log\frac{u_i^N}{u_i^\infty}\bigg)
	- \bigg(\frac{q(u_0^N)}{\bar{q}^N}u_i^N - u_i^N\bigg) \nonumber \\
	&\phantom{xx}+ \bigg(\frac{\bar{q}_i^N}{\bar{q}^N} - u_i^\infty\bigg)\bigg\}. \nonumber
\end{align}
Fix $i\in\{1,\ldots,n\}$. We know from Lemmas \ref{lem.convf} and \ref{lem.conv1} that
$\lim_{N\to\infty} q(u_0^N)/\bar{q}^N=1$ a.e. Together with the boundedness of $u_i^N$,
this shows that
$$
  \lim_{N\to\infty}\bigg(\frac{q(u_0^N)}{\bar{q}^N}u_i^N - u_i^N\bigg) = 0
$$
as well as
\begin{align*}
  \lim_{N\to\infty}&\bigg(\frac{q(u_0^N)}{\bar{q}^N}u_i^N\log\frac{q(u_0^N)u_i^N}{\bar{q}_i^N}
	- u_i^N\log\frac{u_i^N}{u_i^\infty}\bigg) \\
	&= \lim_{N\to\infty}\bigg(\frac{q(u_0^N)}{\bar{q}^N}u_i^N
	\log\frac{(q(u_0^N)/\bar{q}^N)u_i^N}{\bar{q}_i^N/\bar{q}^N}
	- u_i^N\log\frac{u_i^N}{u_i^\infty}\bigg) \\
	&= \lim_{N\to\infty}\bigg\{\frac{q(u_0^N)}{\bar{q}^N}u_i^N
	\log\frac{q(u_0^N)}{\bar{q}^N}
	+ \bigg(\frac{q(u_0^N)}{\bar{q}^N}-1\bigg)u_i^N\log\frac{u_i^N}{u_i^\infty} \\
	&\phantom{xx}{}
	- \frac{q(u_0^N)}{\bar{q}^N}u_i^N\log\frac{\bar{q}_i^N/\bar{q}^N}{u_i^\infty}\bigg\}
	= - \lim_{N\to\infty}\frac{q(u_0^N)}{\bar{q}^N}u_i^N
	\log\frac{\bar{q}_i^N/\bar{q}^N}{u_i^\infty}.
\end{align*}
To show that the limit on the right-hand side equals zero, we observe that, because of mass conservation and dominated convergence,
\begin{align*}
  &\lim_{N\to\infty}\bigg(\frac{\bar{q}_i^N}{\bar{q}^N} - u_i^\infty\bigg)
	= \lim_{N\to\infty}\bigg(\frac{1}{|\Omega|}\int_\Omega\frac{q(u_0^N)}{\bar{q}^N}u_i^N dx
	- u_i^\infty\bigg) \\
	&\phantom{x}
	= \lim_{N\to\infty}\bigg(\frac{1}{|\Omega|}\int_\Omega\frac{q(u_0^N)}{\bar{q}^N}u_i^N dx
	- \frac{1}{|\Omega|}\int_\Omega u_i^0 dx\bigg)
	= \lim_{N\to\infty}\frac{1}{|\Omega|}\int_\Omega\bigg(\frac{q(u_0^N)}{\bar{q}^N}-1\bigg)
	u_i^N dx = 0,
\end{align*}
and this is equivalent to $\lim_{N\to\infty}\log((\bar{q}_i^N/\bar{q}^N)/u_i^\infty)=0$. We conclude that
$$
  \lim_{N\to\infty}\bigg(\frac{q(u_0^N)}{\bar{q}^N}u_i^N\log\frac{q(u_0^N)u_i^N}{\bar{q}_i^N}
	- u_i^N\log\frac{u_i^N}{u_i^\infty}\bigg) = 0.
$$
Putting together the previous limits, we have proved \eqref{3.aux3}.
\end{proof}


\subsection{Convergence of $h^*$}

We conclude from Lemmas \ref{lem.convf} and \ref{lem.key} that
$\lim_{N\to\infty}h_1^*(u^N|u^\infty) = 0$. We claim that also $h_2^*$ and $h_3^*$ converge to zero as $N\to\infty$. Since $u_i^N$ and $u_i^\infty$ are bounded in $[0,1]$, we have the estimate
\cite[Lemma 16]{HJT22}
$$
  \frac12\sum_{i=1}^n(u_i^N-u_i^\infty)^2
	\le \sum_{i=1}^n\bigg(u_i^N\log\frac{u_i^N}{u_i^\infty} - (u_i^N-u_i^\infty)\bigg)
	= h_1^*(u^N|u^\infty)\to 0,
$$
showing that $u_i^N\to u_i^\infty$ a.e.\ in $\Omega\times(0,1]$ as $N\to\infty$ for $i=1,\ldots,n$. We deduce from the continuity of $\chi$ that also $\lim_{N\to\infty}h_3^*(u^N|u^\infty)=0$.

For the limit of $h_2^*$, we observe that $u_0^N=1-\sum_{i=1}^n u_i^N\to u_0^\infty$ a.e. Hence, for any fixed $(x,s)\in\Omega\times(0,1]$, there exists $N_0\in\N$ such that
$1/2\le u_0(x,s+N)/u_0^\infty\le 3/2$ for $N>N_0$. Next, we write $h_2^*$ as
$$
  h_2^*(u^N|u^\infty) = \int_{u_0^\infty}^{u_0^N}\log\frac{q(s)}{q(u_0^\infty)}ds
	= u_0^\infty\int_1^{u_0^N/u_0^\infty}\log\frac{q(u_0^\infty\sigma)}{q(u_0^\infty)}d\sigma.
$$
Since the integrand is a function in $L^1(1/2,3/2)$, it follows from the absolute continuity of the integral that $\lim_{N\to\infty}h_2^*(u^N|u^\infty)=0$ a.e.\ in $\Omega\times(0,1]$. By definition of $h^*$, we have proved that $\lim_{N\to\infty}h^*(u^N|u^\infty)=0$.


\subsection{Convergence in $L^p(\Omega)$}

We deduce from the relative entropy inequality \eqref{3.ei} that $t\mapsto \int_\Omega h^*(u(t)|u^\infty)dx$ is bounded and nonincreasing. Then it follows from the limit $\lim_{N\to\infty}h^*(u^N|u^\infty)=0$ that in fact we have the convergence for all sequences $t\to\infty$,
$\lim_{t\to \infty}\int_\Omega h^*(u(t)|u^\infty)dx=0$ and in particular, since $h_2^*\ge 0$ and $h_3^*\ge 0$,
$$
  \lim_{t\to \infty}\int_\Omega h_1^*(u(t)|u^\infty)dx=0.
$$
Using \cite[Lemma 16]{HJT22} again, we have
$$
  \lim_{N\to\infty}\frac12\sum_{i=1}^n\int_\Omega(u_i(t)-u_i^\infty)^2 dx
	\le \lim_{N\to\infty}\int_\Omega h_1^*(u(t)|u^\infty)dx=0.
$$
The convergence in $L^p(\Omega)$ for any $p<\infty$ then follows from the uniform bound
for $(u_i(t))_{t>0}$, finishing the proof.

\begin{remark}[Drift terms]\label{rem.drift}\rm
Equations \eqref{1.eq2} with drift terms read as
$$
  \pa_t u_i = D_i\diver\bigg\{u_ip_i(u)q(u_0)\na\bigg(\log\frac{u_ip_i(u)}{q(u_0)}
	+ \Phi_i\bigg)\bigg\}, \quad i=1,\ldots,n,
$$
where $\Phi_i=\Phi_i(x)$ are given (electric or environmental) potentials. Adding the associated energy to the entropy density \eqref{1.h},
$$
  h_2(u) = \sum_{i=1}^n(u_i(\log u_i-1)+1) + \int_1^{u_0}\log q(s)ds + \chi(u)
	+ \sum_{i=1}^n u_i\Phi_i,
$$
we can compute (formally) the entropy inequality, giving
$$
  \frac{d}{dt}\int_\Omega h_2(u)dx + \int_\Omega\sum_{i=1}^n D_iu_ip_i(u)q(u_0)
	\bigg|\na\bigg(\log\frac{u_ip_i(u)}{q(u_0)} + \Phi_i\bigg)\bigg|^2 dx = 0.
$$
It was shown in \cite[Section 3.2]{ZaJu17} that the entropy production term with $\Phi_i=0$ can be bounded from below by $p_i(u)(q(u_0)\sum_{i=1}^n|\na\sqrt{u_i}|^2+|\na\sqrt{q(u_0)}|^2)$. Such an estimate seems to be impossible in the presence of $\na\Phi_i$. Indeed, the entropy inequality shows that
\begin{align*}
  4\int_0^\infty&\int_\Omega q(u_0)^2e^{-\Phi_i}
	\bigg|\na\bigg(\frac{u_ip_i(u)e^{\Phi_i}}{q(u_0)}\bigg)^{1/2}
	\bigg|^2 dx \\
	&= \int_0^\infty\int_\Omega u_ip_i(u)q(u_0)\bigg|\na\bigg(\log\frac{u_ip_i(u)}{q(u_0)}
	+ \Phi_i\bigg)\bigg|^2< \infty.
\end{align*}
Thus, in the special case $q(0)>0$ and if $\Phi_i$ is bounded from above,
we conclude the existence of a subsequence $t_k\to \infty$ such that $\na(u_ip_i(u)e^{\Phi_i}/q(u_0))^{1/2}(t_k)\to 0$ strongly in $L^2(\Omega)$ as $k\to\infty$, and one may proceed similarly as in \cite[Section 5]{BFS14}. However, the condition $q(0)=0$ is needed to model correctly the transition rate of nonoccupied cells in the lattice model \cite{BDPS10,ZaJu17}.
\qed\end{remark}


\end{document}